\documentclass[11pt]{amsart}
\usepackage{graphicx, color}
\usepackage{amscd}
\usepackage{amsmath}
\usepackage{amsfonts}
\usepackage{amssymb}
\usepackage{mathrsfs}
\newtheorem{theorem}{Theorem}[section]

\newtheorem{proposition}{Proposition}[section]
\newtheorem{corollary}{Corollary}[section]

\newtheorem{remark}{Remark}[section]
\numberwithin{equation}{section}
\newcommand{\RR}{\mathbb R}
\newcommand{\NN}{\mathbb N}

\begin{document}


\title[Elliptic problems on weighted locally finite graphs]{Elliptic problems on weighted\\ locally finite graphs}

\author[M. Imbesi]{Maurizio Imbesi}
\address[Maurizio Imbesi]{Dipartimento di Scienze Matematiche e Informatiche, Scienze Fisiche e Scienze della Terra\\
Universit\`a degli Studi di Messina, Viale F. Stagno d'Alcontres 31, 98166 Messina, Italy}
\email{\tt maurizio.imbesi@unime.it}

\author[G. Molica Bisci]{Giovanni Molica Bisci}
\address[Giovanni Molica Bisci]{Dipartimento di Scienze Pure e Applicate,
          Universit\`a degli Studi di Urbino `Carlo Bo',
          Piazza della Repubblica 13, 61029 Urbino (Pesaro e Urbino), Italy}
\email{\tt giovanni.molicabisci@uniurb.it}

\author[D.D. Repov\v{s}]{Du\v{s}an D. Repov\v{s}}
\address[Du\v{s}an D. Repov\v{s}]{Faculty of Education, and Faculty of Mathematics and Physics\\
University of Ljubljana \& Institute of Mathematics, Physics and Mechanics\\
1000 Ljubljana, Slovenia}
\email{\tt dusan.repovs@guest.arnes.si}



\begin{abstract}
Let $\mathscr G:= (V,E)$ be a weighted locally finite graph whose finite measure $\mu$ has a positive lower bound. Motivated by  wide interest in the current literature,
in this paper we study the existence of classical solutions for a class
of elliptic equations involving the $\mu$-Laplacian operator on the graph $\mathscr G$, whose analytic expression is given by
\begin{equation*}
\Delta_{\mu} u(x) := \frac{1}{\mu (x)} \sum_{y\sim x} w(x,y) (u(y)-u(x)),\
\hbox{for all}
\
 x\in V
\end{equation*}
where $w : V\times V \rightarrow [0,+\infty)$ is a weight symmetric function and the sum on the right-hand side of the above expression is taken on the neighbors vertices $x,y\in V$, that is $x\sim y$ whenever $w(x,y) > 0$.
More precisely, by exploiting direct variational methods, we study problems whose simple prototype has the following form
\begin{equation*}\label{Np000}
\left\{
\begin{array}{l}
-\Delta_{\mu} u(x)=\lambda f(x,u(x)),\quad x \in \mathop D\limits^ \circ\\
\medskip
\,\, u|_{\partial D}=0,\\
\end{array}
\right.
\end{equation*}
\noindent where $D$ is a bounded domain of $V$ such that $\mathop D\limits^ \circ\neq \emptyset$ and $\partial D\neq \emptyset$, the nonlinear term $f : D \times \RR \rightarrow \RR$ satisfy suitable structure conditions and $\lambda$ is a positive real parameter. By applying a critical point result coming out from a
classical Pucci-Serrin theorem in addition to a local minimum result for differentiable
functionals due to Ricceri, we are able to prove the existence of at least two solutions
for the treated problems. We emphasize the crucial role played by the famous Ambrosetti-Rabinowitz growth condition along the proof of the main theorem and its consequences. Our results improve the general results obtained by  Grigor'yan, Lin, and Yang (J. Differential Equations 261(9) (2016),  4924-4943).
\end{abstract}

\maketitle

{
\small \it
\begin{flushright}
Dedicated to the memory of\\
 Edward Fadell and Sufian Husseini
\end{flushright}
}

\tableofcontents

\section{Introduction}\label{Intro}
Let $\mathscr G:= (V,E)$ be a weighted locally finite graph and $\mu$ be a positive finite measure over $V$ that admits a positive lower bound.
The aim of this paper is to study the existence of non-trivial solutions for certain classes of
nonlinear elliptic Dirichlet one-parameter problems of the form
\begin{equation}\label{Np0}
\left\{
\begin{array}{l}
-\Delta_{\mu} u(x)=\alpha(x)u(x)+\lambda f(x,u(x))\quad x \in \mathop D\limits^ \circ\\
\medskip
\,\, u|_{\partial D}=0,
\end{array}
\right.
\end{equation}
where $D$ is a bounded domain of $V$ such that $\mathop D\limits^ \circ\neq \emptyset$ and $\partial D\neq \emptyset$ and $\lambda$ is a real positive parameter. Moreover, the coefficient $\alpha:V\rightarrow \RR$ of the linear term is assumed integrable over $D$ and $f:D\times \RR\rightarrow \RR$ is the nonlinear part, which is continuous for every $x\in D$
and satisfies some growth restrictions both near zero and at infinity.

There is an extensive theory for the study of nonlinear elliptic equations
on Euclidean domains by using Sobolev spaces and related Sobolev embedding results. A natural question
arises of how to establish an appropriate framework to cope with \eqref{Np0} on
graph domains.

In this new setting several difficulties naturally arises. For instance, the standard
concept of generalized derivative of a function cannot be used, and so the notion of
differential operators such as the Laplacian on graph domains need to be clarified.
More precisely, according to the notations and definitions given in Section \ref{Preliminaries}, the $\mu$-Laplacian operator $\Delta_\mu:W^{1,2}_0(D)\rightarrow L^2(D)$ has
been defined in the distributional sense
\begin{equation*}
 \begin{split}
 \int_D\Delta_\mu u(x)v(x)d\mu=-\int_D\Gamma (u,v)(x)d\mu,
 \end{split}
 \end{equation*}
  for every $v\in W^{1,2}_0(D),$ where the map $\Gamma:W^{1,2}_0(D)\times W^{1,2}_0(D)\rightarrow \RR^{D}$
is the gradient form whose analytical local expression is given below
\begin{equation*}\label{problem2G}
\Gamma (u,v)(x) := \frac{1}{2\mu (x)} \sum_{y\sim x} w(x,y) (u(y)-u(x))(v(y)-v(x)),
\end{equation*}
for every $x\in D$.\par

Once a Laplacian is defined, we may
introduce a Hilbert space structure and then establish compactness theorems allowing
\eqref{Np0} to
be investigated.

In this direction, A. Grigor'yan in \cite{G1,G2} and A. Grigor'yan, Y. Lin, and Y. Yang in \cite{GLY2,GLY3} obtained some embedding theorems that can be applied by studying
variational problems settled on graphs for which Poincar\'{e}-type inequalities hold; see, for instance, the paper \cite{Z}.
Moreover, the $\mu$-Laplacian operator or some of its generalizations as described in Section \ref{YG} permit to describe Random walks on graphs. A comprehensive discussion
can be found in the monograph \cite{Lbook}; see also the papers \cite{CG} and \cite{LW,LY} for related topics.
\indent Moreover, from the parabolic point of view, M. Barlow, T. Coulhon, and A. Grigor'yan in \cite{BCG} studied some upper estimates on the long time behavior of the heat
kernel on non-compact Riemannian manifolds and infinite graphs, which
only depend on a lower bound of the volume growth.\par

Set
\begin{align}\label{lamba1}
\lambda_1:=\inf_{u\in C_0(D)\setminus\{0\}}\frac{\displaystyle \int_{D} |\nabla u|^2(x) d\mu}{\displaystyle \int_{D} |u(x)|^2 d\mu}.
\end{align}

The main result of the present paper is a multiplicity theorem as stated here below.

\begin{theorem}\label{MolicaBisciPrincipal}
Let $\mathscr G:=(V,E)$ be a weighted locally finite graph, $D$ be a bounded domain of $V$ such that $\mathop D\limits^ \circ\neq \emptyset$ and $\partial D\neq \emptyset$ and let $\mu:D\rightarrow ]0,+\infty[$ be a measure on $D$.
Let $\alpha\in L^1(D)$ be a function satisfying either
\begin{equation}\label{alfa1}
\alpha(x)\leq 0\,\,\, \mbox{for every}\,\,\, x\in D
\end{equation}
or
\begin{equation}\label{alfa2}
\int_D|\alpha(x)|d\mu<\mu_0^2\lambda_1\mbox{ with }\,\, \mu_0:=\min_{x\in D}\mu(x)>0
\end{equation}
and let
$f:D\times\RR\rightarrow\RR$ be a function such that
\begin{equation}\label{f0}
f(x,\cdot)\,\,\mbox{is continuous in $\RR$ and } f(x,0)\neq 0\,\,\mbox{for some}\,\, x\in D
\end{equation}
as well as
\begin{equation}\label{f2}
\begin{aligned}
& \qquad\,\,\qquad \mbox{there are}\,\, \beta>2\,\, \mbox{and}\,\, r_0>0\,\, \mbox{such that}\\
& tf(x,t)\geq \beta F(x,t)>0\,\, \mbox{for any}\,\, |t|\geq r_0\,\,\mbox{and every}\,\, x\in D,
\end{aligned}
\end{equation}
where $F$ is the potential given by
\begin{equation}\label{F}
{\displaystyle F(x,t):=\int_0^t f(x,s)\,ds}\quad  \mbox{for any}\,\,(x,t)\in D\times\RR\,.
\end{equation}

Then  for any $\varrho>0$ and any
\begin{equation}\label{lambda}
0<\lambda<\frac{\varrho}{2\displaystyle\max_{{\footnotesize\begin{array}{c}
x\in D\\
|s|\leq \kappa\sqrt{\varrho}
\end{array}}}\left|\int_0^{s}f(x,t)dt\right|}\,,
\end{equation}
where
\begin{equation}\label{kappa}
\kappa:=\left\{\begin{array}{ll}
\displaystyle\frac{1}{\mu_0\sqrt{\lambda_1}} & \mbox{if \eqref{alfa1} holds}\\
\\
{\displaystyle \frac{1}{\mu_0\sqrt{\lambda_1}\sqrt{1-\displaystyle\frac{1}{\mu_0^2\lambda_1}{\displaystyle\int_D{|\alpha(x)|}d\mu}}}} & \mbox{if \eqref{alfa2} holds},
\end{array}\right.
\end{equation}
 problem~\eqref{Np0} admits at least two non-trivial solutions one of which lies in
$$
\mathbb{B}_\varrho:=\left\{u\in W^{1,2}_0(D):{\displaystyle \langle u,u \rangle -\int_D{\alpha(x)}|u(x)|^{2}d\mu}<{\varrho}\right\}.
$$
\end{theorem}

We point out that the maximal interval of $\lambda$'s, where the conclusion of Theorem~\ref{MolicaBisciPrincipal} holds,  is given by $(0, \lambda^*)$, where
\begin{equation}\label{struttela}
\lambda^*  :=\frac 1 2 \sup_{\varrho>0}\frac{\varrho}{\displaystyle\max_{{\footnotesize\begin{array}{c}
x\in D\\
|s|\leq \kappa\sqrt{\varrho}
\end{array}}}\left|\int_0^{s}f(x,t)dt\right|}\in (0,+\infty]
\end{equation}
\noindent with $\kappa$  as in \eqref{kappa}.
\begin{remark}\label{Una}
We emphasize that Theorem \ref{MolicaBisciPrincipal} ensures the existence of one non-trivial solution
(instead of two) if we require that
\begin{equation}\label{f02}
f(x,\cdot)\,\,\mbox{is continuous in $\RR$ and } f(x,0)= 0\,\,\mbox{for any}\,\, x\in \mathop D\limits^ \circ,
\end{equation}
instead of condition \eqref{f0}.
\end{remark}
\indent
The existence result given in Theorem \ref{MolicaBisciPrincipal} is a consequence of \cite[Theorem 1]{puse} and \cite[Theorem 6]{R2}; see Theorem \ref{Pucci-Serrin+Ricceri} below. More precisely, the main result is achieved by proving that the geometry of Theorem \ref{Pucci-Serrin+Ricceri} is respected by our abstract framework: for this we
develop a functional analytical setting in order to correctly encode the datum in the variational
formulation. We emphasize that the compactness condition required by Theorem \ref{Pucci-Serrin+Ricceri} is satisfied in the graph-theoretical setting thanks to the choice of the
functional setting we work in; see Section \ref{Preliminaries}.
Moreover, Theorem \ref{MolicaBisciPrincipal} improves the existence results obtained in \cite[Theorem 2]{GLY2} and it can be viewed as a graph-theoretical counterpart of \cite[Theorem 4]{R0}.\par

We also observe that, when we deal with partial differential equations driven by the
Laplace operator (or, more generally, by uniformly elliptic
operators) with homogeneous Dirichlet boundary conditions,
assumption \eqref{f2} - namely the Ambrosetti-Rabinowitz condition (briefly (AR)) - is the standard superquadraticity condition
on $F$; see, among others, the classical papers \cite{ar, rabinowitz, struwe}. This condition is often considered when dealing with superlinear elliptic boundary value problems and its importance is due to the fact that \eqref{f2} assures the boundedness of the
Palais-Smale sequences for the energy functional associated with the problem under
consideration.

A special case of Theorem~\ref{MolicaBisciPrincipal} reads as follows.
\begin{theorem}\label{MolicaBisciSpecial}
Let $\mathscr G:=(V,E)$ be a weighted locally finite graph, $D$ a bounded domain of $V$ such that $\mathop D\limits^ \circ\neq \emptyset$ and $\partial D\neq \emptyset$ and let $\mu:D\rightarrow ]0,+\infty[$ be a measure on $D$.
Let $\alpha\in L^1(D)$ satisfy \eqref{alfa1} and let
$f:D\times\RR\rightarrow\RR$ be a function satisfying \eqref{f0} and \eqref{f2}. Moreover, assume that
\begin{equation}\label{f1}
\begin{aligned}
& \mbox{there are positive constants}\,\,\, M_0\,\, \mbox{and}\,\, \sigma\,\, \mbox{such that}\\
& \,\,\,\,\,\,\,\displaystyle\max_{(x,s)\in \mathop D\limits^ \circ\times[-M_0,M_0]}|f(x,s)|\leq \mu_0^2\frac{M_0}{(\sigma+1)}\frac{\lambda_1}{2},
\end{aligned}
\end{equation}
where $\mu_0:=\displaystyle\min_{x\in D}\mu(x)>0$.\par
Then  the Dirichlet problem
\begin{equation}\label{Np}
\left\{
\begin{array}{l}
-\Delta_{\mu} u(x)=\alpha(x)u(x)+ f(x,u(x))\quad x \in \mathop D\limits^ \circ\\
\medskip
\,\, u|_{\partial D}=0,\\
\end{array}
\right.
\end{equation}
admits at least two non-trivial solutions one of which lies in the open ball $\mathbb{B}_{\mu_0^2M_0^2\lambda_1}$.
\end{theorem}

\indent On account of Remark \ref{Una}, Theorem \ref{MolicaBisciSpecial} ensures that problem~\eqref{Np0} admits at least one non-trivial solutions requiring \eqref{f02} instead of \eqref{f0}.
The main novelty here, with respect to the approach considered in \cite[Theorem 2]{GLY2}, is to avoid the following condition of the nonlinearity term $f$ at zero
\begin{equation}\label{f1l}
\limsup_{t\rightarrow 0^+}\frac{f(x,t)}{t}<\lambda_1,
\end{equation}
for every $x\in D$; see Corollary \ref{MolicaBisciSpecialSpecial}.  Moreover, we also emphasize that Theorems 1 and 2 in \cite{GLY2} are an immediate and direct consequence of Theorem \ref{MolicaBisciSpecial}; see Theorems \ref{MolicaBisciSpecial2} and \ref{MolicaBisciSpecial2C}. Indeed, let us consider a relaxed version of condition \eqref{f2} as given below
\begin{equation}\label{f2new}
\begin{aligned}
&\qquad \qquad \mbox{there are}\,\, \beta>2\,\, \mbox{and}\,\, r_0>0\,\, \mbox{such that}\\
& tf(x,t)\geq \beta F(x,t)>0\,\, \mbox{for any}\,\, t\geq r_0\,\, \mbox{and every}\,\,\, x\in D.
\end{aligned}
\end{equation}

By using \eqref{f2new} the following result holds; see \cite[Theorem 1]{GLY2}.

\begin{theorem}\label{MolicaBisciSpecial2}
Let $\mathscr G:=(V,E)$ be a weighted locally finite graph, $D$ a bounded domain of $V$ such that $\mathop D\limits^ \circ\neq \emptyset$ and $\partial D\neq \emptyset$ and let $\mu:D\rightarrow ]0,+\infty[$ be a measure on $D$.
Let $\alpha\in L^1(D)$ satisfy \eqref{alfa1} and let
$f:D\times\RR\rightarrow[0,+\infty[$ be a function such that \eqref{f02} holds.
 Moreover, assume that \eqref{f1l} and \eqref{f2new} hold.
Then  problem \eqref{Np}
admits at least one non-trivial and non-negative solution.
\end{theorem}

A particular form of Theorem \ref{MolicaBisciSpecial2} reads as follows; see \cite[Theorem 2]{GLY2}.
\begin{theorem}\label{MolicaBisciSpecial2C}
Let $\mathscr G:=(V,E)$ be a weighted locally finite graph, $D$ a bounded domain of $V$ such that $\mathop D\limits^ \circ\neq \emptyset$ and $\partial D\neq \emptyset$ and let $\mu:D\rightarrow ]0,+\infty[$ be a measure on $D$.
Let $\gamma,p\in \RR$ with $\gamma<\lambda_1$
and $p>2$.
Then  the following problem
\begin{equation}\label{NpG}
\left\{
\begin{array}{l}
-\Delta_{\mu} u(x)=\gamma u(x)+ |u(x)|^{p-2}u(x)\quad x \in \mathop D\limits^ \circ\\
\medskip
\,\, u|_{\partial D}=0,\\
\end{array}
\right.
\end{equation}
admits at least one positive solution.
\end{theorem}

Suppose that $h : V \rightarrow\RR$ is a function satisfying the coercive
condition on $D$, namely there exists some constant $\delta>0$ such that
\begin{equation}\label{f1lX}
\int_{D}u(x)(-\Delta_\mu +h)u(x)d\mu\geq \delta \int_{D} |\nabla u|^2(x) d\mu,
\end{equation}
for every $u:V\rightarrow\RR$. By using \eqref{f1lX} and the classical (AR) condition, semi-linear elliptic equations of the form
$$
-\Delta_\mu u+h(x)u(x)=|u(x)|^{p-2}u(x)\quad x \in D
$$
on weighted locally finite graphs have been studied in \cite[Theorem 1.1]{Z}. For the sake of completeness we point out that in Theorem \ref{MolicaBisciSpecial2C} we don't require any coercivity assumption as in \eqref{f1lX}.

Finally, an immediate meaningful consequence of Theorem \ref{MolicaBisciSpecial} is given below;
 see Remark \ref{RemFinale} for additional comments and details.

\begin{corollary}\label{MolicaBisciSpecialSpecial}
Let $\mathscr G:=(V,E)$ be a weighted locally finite graph and $D$ be a bounded domain of $V$ such that $\mathop D\limits^ \circ\neq \emptyset$ and $\partial D\neq \emptyset$.
Furthermore, let
$f:\RR\rightarrow[0,+\infty[$ be a continuous function and
$\mu:D\rightarrow]0,+\infty[$ be a measure on $D$ such that
\begin{equation}\label{NpXXss}
\mu_0>2\sqrt{\frac{\displaystyle\max_{t\in [-1,1]}f(t)}{\lambda_1}},
\end{equation}
where $\mu_0:=\displaystyle\min_{x\in D}\mu(x)$.\par
\indent Then  if $f(x,0)=0$ for every $x\in D$ and \eqref{f2new} holds, the Dirichlet problem
\begin{equation}\label{NpXX}
\left\{
\begin{array}{l}
-\Delta_{\mu} u(x)=f(x,u(x))\quad x \in \mathop D\limits^ \circ\\
\medskip
\,\, u|_{\partial D}=0,\\
\end{array}
\right.
\end{equation}
admits at least a non-trivial and non-negative solution.\par
If $f(x,0)\neq 0$ for some $x\in \mathop D\limits^ \circ$ and \eqref{f2} holds, then problem \eqref{NpXX}
admits at least two non-trivial solutions one of which lies in the open ball $\mathbb{B}_{\mu_0^2\lambda_1}$.
\end{corollary}

The paper is organized as follows. In Section \ref{Preliminaries} we recall some basic notions on weighted locally finite graphs. Successively, Section \ref{SezA} is devoted to the variational methods and abstract framework that we use here. In Section \ref{EMresult} the main existence and multiplicity results have been proved. Finally, in the last section the extension of the main results to Yamabe-type equations of weighted locally finite graphs and involving the $(m,p)$-Laplacian operator have been considered; see Theorem \ref{MolicaBisciPrincipalGeneralX}.\par For the sake of completeness we cite the very recent monograph \cite{G2} by A. Grigor'yan as general reference for basic notions used throughout the manuscript. See also the papers \cite{CGY1, CGY2} for interesting problems and results.

\section{Analysis on weighted locally finite graphs}\label{Preliminaries}

Let $V$ be a nonempty set and let $E\subset V \times V$ be a binary relation. Set
$$
x\sim y \quad \mbox{ in } V \quad \Leftrightarrow\quad xy:=(x,y)\in E
$$
and assume that
$$
xy\in E\quad \mbox{ if and only if }\quad yx\in E.
$$
\indent The couple $\mathscr G:= (V,E)$ is said to be a non-oriented {graph} (briefly graph) with set of vertices $V$ a set of edges $E$.
If every $x\in V$ has finitely many neighbors, that is
$$|\{y \in V:  xy\in E\}| < +\infty,$$
the graph $\mathscr G$ is said to be {weighted locally finite}.\par
\indent A weight on a weighted locally finite graph $\mathscr G$ is a function $w : V\times V \rightarrow [0,+\infty[$ such that
$$
w(x,y)=w(y,x)\quad{\rm  and} \quad w(x,x)=0,
\
\hbox{for all}
\
 x,y\in V.
$$
In such a case
$$
x\sim y \,\,\,\mbox{ in } V \quad \mbox{ if and only if }\quad w(x,y)>0,
$$
and, for every fixed $x\in V,$ its degree, defined by $$\deg(x):=\displaystyle\sum_{y\in V} w(x,y),$$ is finite.\par

\indent From now on, let us restrict our attention on a weighted locally
finite graph $\mathscr G$ and let $\mu:V\rightarrow ]0,+\infty[$ be a finite positive measure on $V$.
Moreover, let us define
\begin{equation}\label{IntegraqlX}
\int_V u(x)d\mu := \sum_{x\in V} \mu(x)u(x),
\end{equation}
for every $u:V\rightarrow \RR$. With abuse of notations classical notations for Lebesgue spaces are still used.\par
\indent Fixed $n\in \mathbb{N}^*$, a path on $\mathscr G$ is any finite sequence of vertices $(x_k)_{k=1}^{n}\subseteq V$ such that
$$
x_{k}x_{k+1}\in E\,\,\,\mbox{ for every }\,\,\, k=1,...,n-1.
$$
\noindent The length of a path on $\mathscr G$ is the number of edges in the path.
Thus, if $x,y\in V$ we denote by $P(x,y)$ be a path connecting $x$ with $y$, i.e.
a path on $\mathscr G$ in which $x_1=x$ and $x_n=y$. Finally, $\wp(x,y)$ denotes the
set of paths connecting $x$ with $y$, and $\ell(P(x,y))$ is the length of $P(x,y)\in \wp(x,y).$\par

We say that $\mathscr G$ is connected
if, for any two vertices $x,y\in V,$ there is a path connecting them. Any connected graph $\mathscr G$ has a natural metric structure induced by the metric $d:V\times V\rightarrow [0,+\infty[$ given by
\begin{equation}\label{problemd}
d(x,y):=\left\{ \begin{array}{ll}
\displaystyle\min_{P(x,y)\in \wp(x,y)}\{\ell(P(x,y))\} & \mbox{if} \, \ \ \ x\neq y\\
0 & \mbox{if} \, \ \ \  x=y. \end{array} \right.
\end{equation}
Consequently, a connected weighted locally finite
graph has at most countable many vertices.

\indent  As customary, given $\mathscr G:=(V,E)$ be a weighted locally finite graph, a (proper) subgraph $\mathscr H$ of $\mathscr G$ is a graph of the form $\mathscr H:=(D, L)$, where $D\subset V$ and $L\subset E.$
Moreover, the vertex boundary $\partial D$ and the
vertex interior
$\mathop D\limits^ \circ$ of a connected subgraph $\mathscr H:=(D, L)$ are defined respectively as
$$
\partial D:=\{x\in D:\exists y\not\in D \mbox{ such that } xy\in E\},
$$
and
$
\mathop D\limits^ \circ:=D\setminus \partial D.
$ In addition, we say that $D$
 is bounded if it is a bounded subset of $V$ with respect to the usual vertex
distance given in \eqref{problemd}.\par

\indent For every function $u : V \rightarrow \RR$, the $\mu$-Laplacian (or Laplacian for short) of $u$ is defined as
\begin{equation}\label{probl1}
\Delta_\mu u(x) := \frac{1}{\mu (x)} \sum_{y\sim x} w(x,y) (u(y)-u(x)),
\end{equation}
for every $x\in V$. Moreover, the associated gradient form has the expression given below
\begin{equation}\label{problem2}
\Gamma (u,v)(x) := \frac{1}{2\mu (x)} \sum_{y\sim x} w(x,y) (u(y)-u(x))(v(y)-v(x)),
\end{equation}
for every $x\in V$. Hence, for every $x\in V,$ we denote by
\begin{equation}\label{problem3}
|\nabla u|(x) := \sqrt{\Gamma(u)(x)} = \left(\frac{1}{2\mu (x)} \sum_{y\sim x} w(x,y) (u(y)-u(x))^2\right)^{1/2},
\end{equation}
the slope of the function $u,$ where $\Gamma(u) := \Gamma(u,u)$.\par

By using \eqref{probl1} and \eqref{problem3}, if $D$ is a bounded and connected set of $V$ (briefly a domain) such that $\mathop D\limits^ \circ\neq \emptyset$ and $\partial D\neq \emptyset$, the Sobolev space associated to $D$ can be defined as follows
\begin{equation}\label{Sobolev}
W^{1,2}(D):=\{u:D\rightarrow \RR: \mbox{$u$ is a real-valued function}\},
\end{equation}
endowed by the norm
\begin{equation}\label{nu0}
 \|u\|_{W^{1,2}(D)} := \|u\|_{L^2(D)}+\||\nabla u|\|_{L^2(D)},
 \end{equation}
for every $u\in W^{1,2}(D)$. Hence, the space $W^{1,2}(D)$ coincides with
the set of all real-valued functions defined on the domain $D.$ \par
Now, it easily seen that $W^{1,2}(D)$ admits a natural structure of Hilbert space  in which the inner product is defined as follows
\begin{equation}\label{psc}
 \langle u,v \rangle := \int_{D} \Gamma(u,v)(x)d\mu+\int_{D}u(x)v(x)d\mu,
  \
\hbox{for all}
\
  u,v\in W^{1,2}(D).
 \end{equation}

\noindent Moreover, since
 $D$ is a finite set, we emphasize that the
Hilbert space $W^{1,2}(D)$ is finite dimensional and the finiteness of $D$ also yields
\begin{equation}\label{mu0}
\mu_0:=\min_{x\in D}\mu(x)>0,
\end{equation}
since the measure $\mu$ is positive on $V.$\par
Finally, $W^{1,2}_0(D)$ is the
completion of
\begin{equation}\label{c0}
C_0(D):=\{u:D\rightarrow \RR: u=0 \mbox{ on } \partial D\},
\end{equation}
 with respect to the Sobolev norm $\|\cdot\|_{L^2(D)}+\||\nabla \cdot|\|_{L^2(D)}$, i.e. $$W^{1,2}_0(D)=\overline{C_0(D)}^{\|\cdot\|_{L^2(D)}+\||\nabla \cdot|\|_{L^2(D)}}.$$

 Now, a sort of integration formula by parts is valid. More precisely
 \begin{equation}\label{parti}
 \begin{split}
 -\int_D(\Delta_\mu u(x))v(x)d\mu=&-\int_D\left(\frac{\sum_{y\sim x} w(x,y) (u(y)-u(x))}{\mu(x)}\right) v(x)d\mu\\
 =&\int_D\frac{1}{2\mu (x)} \sum_{y\sim x} w(x,y) (u(y)-u(x))(v(y)-v(x))d\mu,
 \end{split}
 \end{equation}
 for every $v\in W^{1,2}_0(D);$
 see \cite{HK} for additional comments and details.

\indent Let us set
\begin{equation*}\label{eq21}
\|\cdot\|_{L^{\infty}(D)}:=\max_{x\in D}|\cdot|\quad
\mbox{and}
\quad
\|\cdot\|_{L^{\nu}(D)}:=\left(\int_{D} |\cdot|^{\nu} d\mu \right)^{1/\nu},
\end{equation*}
for every $\nu\in [1,+\infty[$ and let $\lambda_1$ be the first eigenvalue of $-\Delta_\mu$ with respect to Dirichlet boundary condition as defined in \eqref{lamba1}.
Clearly $\lambda_1\in ]0,+\infty[,$ since $D$ is finite and thanks to the definition given in \eqref{IntegraqlX}.

 The following compact Sobolev embedding result holds.

\begin{proposition}\label{theorem 7}
Let $\mathscr G:=(V,E)$ be a weighted locally finite graph and $D$ be a bounded domain of $V$ such that $\mathop D\limits^ \circ\neq \emptyset$ and $\partial D\neq \emptyset$. Then $W_0^{1,2}(D)$ is embedded in $L^q(D)$ for every $q\in [1,+\infty]$.\par In particular,
\begin{equation}\label{eq21}
\|u\|_{L^{\infty}(D)}\leq \frac{1}{\mu_0\sqrt{\lambda_1}} \left( \int_{D} |\nabla u|^2(x) d\mu \right)^{1/2},
\end{equation}
and
\begin{equation}\label{eq22}
\|u\|_{L^{\nu}(D)} \leq \frac{\mu(D)^{1/\nu}}{\mu_0\sqrt{\lambda_1}} \left( \int_{D} |\nabla u|^2(x) d\mu \right)^{1/2},
\end{equation}
for every $\nu\in [1,+\infty[$ and every $u\in W_0^{1,2}(D)$.\par
 Moreover, for every bounded sequence $(u_k)_k$ in $W_0^{1,2}(D),$ up to a subsequence, there exists $u\in L^q(D)$ such that $u_k \rightarrow u$ in $L^q(D)$, namely the embedding
$$
W_0^{1,2}(D)\hookrightarrow L^q(D)
$$
is compact provided that $q\in [1,+\infty]$.
\end{proposition}
\begin{proof}
Since $D$ is a finite set, $W_0^{1,2}(D)$ is a finite dimensional space. Consequently $W_0^{1,2}(D)$ is
pre-compact, i.e. for every bounded sequence $(u_k)_k$ in $W_0^{1,2}(D),$ up to a subsequence, there exists $u_\infty\in W_0^{1,2}(D)$ such that $u_k \rightarrow u_\infty$ in $W_0^{1,2}(D)$. In addition
\begin{equation}\label{eq3}
\|u\|_{W_0^{1,2}(D)} = \left( \int_{D} |\nabla u|^2(x) d\mu \right)^{1/2}
\end{equation}
is a norm equivalent to \eqref{nu0} on $W_0^{1,2}(D)$. Hence, on account of \eqref{lamba1}, one has
\begin{equation}\label{eq4}
\left(\int_{D} |u(x)|^2 d\mu \right)^{1/2} = \left(\sum_{x\in D} \mu (x) |u(x)|^2 \right)^{1/2} \leq \frac{1}{\sqrt{\lambda_1}}\left( \int_{D} |\nabla u|^2(x) d\mu \right)^{1/2},
\end{equation}
for any $u\in W_0^{1,2}(D)$.\par
Now, by (\ref{eq4}) and \eqref{mu0} one has that
\begin{equation}\label{eq4p}
 \|u\|_{L^{\infty}(D)} \leq \frac{1}{\mu_0\sqrt{\lambda_1}} \|u\|_{W_0^{1,2}(D)},
\end{equation}
for every $u\in W_0^{1,2}(D)$. Moreover, inequality \eqref{eq4p} immediately yields
\begin{equation}\label{eq4pp}
 \left(\int_{D} |u(x)|^{\nu} d\mu \right)^{1/\nu} \leq \frac{\mu(D)^{1/\nu}}{\mu_0\sqrt{\lambda_1}}\|u\|_{W_0^{1,2}(D)},
\end{equation}
for every $\nu\in [1,+\infty[,$
where $\displaystyle\mu(D):=\sum_{x\in D} \mu (x)$ denotes the volume of $D$. Therefore (\ref{eq21}) and (\ref{eq22}) hold.\par
Finally, since $W_0^{1,2}(D)$ is pre-compact, by (\ref{eq21}) and (\ref{eq22}) the embedding
$
W_0^{1,2}(D)\hookrightarrow L^q(D)
$
is compact as claimed provided that $q\in [1,+\infty]$.
\end{proof}

See \cite[Theorem 7]{GLY2} for additional comments and remarks. We also emphasize that, since $D$ is a finite set, $(W_0^{1,2}(D), \|\cdot\|)$ is a finite dimensional Hilbert space,
where the norm
\begin{equation}\label{eq31}
\|u\|:= \left( \int_{D} |\nabla u|^2(x) d\mu \right)^{1/2}
\end{equation}
on $W_0^{1,2}(D)$ is equivalent to the norm $\|\cdot\|_{W^{1,2}(D)}$ given in \eqref{nu0} thanks to Proposition \ref{theorem 7}. Moreover, in this finite setting, one has
$$
W_0^{1,2}(D)=C_0(D),
$$
where $C_0(D)$ is defined in \eqref{c0}.

\section{Variational methods and abstract framework}\label{SezA}

The aim of this section is to prove that, under natural assumptions on the nonlinear term~$f$, problem~\eqref{Np0} admits two non-trivial classical solutions. As we already said, this is done by means of variational techniques.

Indeed, if $\alpha\in L^1(D)$ and $f:D\times \RR\rightarrow\RR$ is a function such that $f(x,\cdot)$ is continuous in $\RR$ for every $x\in D$, then problem~\eqref{Np0} is of variational nature, and its energy functional $\mathcal J:W_0^{1,2}(D)\rightarrow\RR$ is defined by
\begin{equation}\label{Fu2}
\begin{aligned}
\mathcal J_\lambda(u)& :=\frac{\langle u,u \rangle}{2\lambda}-\frac{1}{2\lambda}\int_D {\alpha(x)}|u(x)|^{2}\,d\mu
-\int_D F(x, u(x))d\mu,
\end{aligned}
\end{equation}
for every $u\in W^{1,2}_0(D)$.\par
\indent Direct computations ensure that the functional $\mathcal{J}_{\lambda}$ is continuously G\^{a}teaux differentiable at $u\in W_0^{1,2}(D)$ and its G\^{a}teaux derivatives at $u\in W_0^{1,2}(D)$ has the form
$$
\langle \mathcal{J}'_{\lambda}(u), v\rangle = \frac{\langle u,v \rangle}{\lambda}-\frac{1}{\lambda}\int_D\alpha(x)u(x)v(x)d\mu
-\int_D f(x,u(x))v(x)d\mu,
$$
for any $v \in W_0^{1,2}(D)$. As usual, a function $u\in W_0^{1,2}(D)$ such that
$\langle \mathcal{J}'_{\lambda}(u), v\rangle =0,$ for every $v \in W_0^{1,2}(D),$
is said to be a critical point of the functional $\mathcal{J}_{\lambda}$.\par

Now, with a fixed $\lambda>0$ in $\RR,$ a function $u\in W^{1,2}_0(D)$ is a weak solution of the problem~\eqref{Np0} if
\begin{equation}\label{wsm}
\frac{\langle u,v \rangle}{\lambda}-\frac{1}{\lambda}\int_D\alpha(x)u(x)v(x)d\mu
-\int_D f(x,u(x))v(x)d\mu=0,
\end{equation}
for any $v \in W_0^{1,2}(D)$.\par
 Thus, the critical points of $\mathcal{J}_{\lambda}$ are exactly the weak solutions of problem~\eqref{Np0}. Moreover, the following result ensures that every weak solution of problem~\eqref{Np0} classically solves it, that is
\begin{equation}\label{NpC}
\left\{
\begin{array}{l}
-\Delta_{\mu} u(x)=\alpha(x)u(x)+ \lambda f(x,u(x))\quad x \in \mathop D\limits^ \circ\\
\medskip
\,\, u|_{\partial D}=0.\\
\end{array}
\right.
\end{equation}

\begin{proposition}\label{strong}
Let $\alpha\in L^1(D)$ and $f:D\times \RR\rightarrow\RR$ be a function such that $f(x,\cdot)$ is continuous in $\RR$ for every $x\in D$. If $u\in W^{1,2}_0(D)$ is a weak solution of problem~\eqref{Np0}, then \eqref{NpC} holds.
\end{proposition}

\begin{proof}
Fixed $\lambda>0$, let $u \in W_0^{1,2}(D)$ such that \eqref{wsm} holds for every $v \in W_0^{1,2}(D)$. Hence, fixed $y\in D,$ let us define the function $v_y\in W_0^{1,2}(D)$ as follows
\begin{equation}\label{problemdy}
v_y(x):=\left\{ \begin{array}{ll}
\displaystyle  (-\Delta_{\mu} u(y)-\alpha(y)u(y)- \lambda f(y,u(y)))\delta_y^{x} & \mbox{if} \, \ \ \ x\not\in \partial D\\
0 & \mbox{if} \, \ \ \  x\in \partial D,\end{array} \right.
\end{equation}
for every $x\in D$, where as customary $\delta_y^{x}$ denotes the Kronecker delta. Since \eqref{wsm} holds for every $v \in W_0^{1,2}(D)$, for $v:=v_y$ by \eqref{psc} and on account of \eqref{parti} it follows that
$$
\label{Integraql}
\int_D(-\Delta_{\mu} u(x)-\alpha(x)u(x)-\lambda f(x,u(x))v_y(x)d\mu=0.
$$
Consequently, definition \eqref{IntegraqlX} immediately yields
$$
\sum_{x\in D}(-\Delta_{\mu} u(x)-\alpha(x)u(x)-\lambda f(x,u(x))v_y(x)\mu(x)=0,
$$
that is
$$
\displaystyle  -\Delta_{\mu} u(y)-\alpha(y)u(y)- \lambda f(y,u(y))=0,
$$
thanks to \eqref{problemdy} and the range of the measure $\mu$. The conclusion immediately follows since $y\in D$ is arbitrary.
\end{proof}

In order to state Theorem \ref{Pucci-Serrin+Ricceri} below we also recall that a $C^1$-functional $J:E\to\RR$, where $E$ is a real Banach
space with topological dual $E^*$, satisfies the \emph{Palais-Smale condition} (in short $(\rm PS)$-condition) when
$$\begin{aligned}
& \quad\quad \emph{Every sequence $(u_k)_k$ in $E$ such that
$(J(u_k))_k$ is bounded and}\\
& \emph{$\|J'(u_k)\|_{E^*}\to 0$ as $k\rightarrow +\infty$
possesses a convergent subsequence in $E$.}
\end{aligned}$$

The abstract tool used along the present paper in order to prove the existence of weak solutions for \eqref{Np0} is the following theorem; see \cite[Theorem~6]{R0}.
\begin{theorem}\label{Pucci-Serrin+Ricceri}
Let $E$ be a reflexive real Banach space and let $\Phi,\Psi:E\to\RR$
be two continuously G\^{a}teaux differentiable functionals such that
\begin{itemize}
\item[] $\Phi$ is
sequentially weakly lower semicontinuous and coercive in $E$
\item[] $\Psi$
is sequentially weakly continuous in $E$.
\end{itemize}
In addition, assume that for each $\mu>0$ the functional $J_\mu:=\mu\Phi-\Psi$ satisfies the $(\rm PS)$-condition. Then  for each $\varrho>\displaystyle\inf_E\Phi$ and each
$$\mu>\inf_{u\in\Phi^{-1}(]-\infty,\varrho[)}
\frac{\displaystyle\sup_{v\in\Phi^{-1}(]-\infty,\varrho[)}\Psi(v)-\Psi(u)}{\varrho-\Phi(u)},$$
the following alternative holds: either the functional $J_\mu$ has a strict global minimum which lies in $\Phi^{-1}(]-\infty,\varrho[)$, or $J_\mu$ has at least two critical points one of which lies in $\Phi^{-1}(]-\infty,\varrho[)$.
\end{theorem}

The abstract resut given in Theorem~\ref{Pucci-Serrin+Ricceri} is an interplay between the celebrated three critical point theorem given by P. Pucci and J. Serrin in \cite{puse} and a quoted local minimum result obtained by B. Ricceri in  \cite{R2}. Some applications of Ricceri's variational principle are contained in \cite{MR,Ri1,R3}. We also cite \cite{k2} for related topics on the variational methods used in this paper. Classical notions can be found in~\cite{brezis}.

The $(\rm PS)$-condition is one of the main compactness assumption required on the energy functional when considering critical point theorem. In order to simplify its proof, in the sequel we will perform the following result, which is valid for the energy functional~$\mathcal J_\lambda$ given in \eqref{Fu2}.
\begin{proposition}\label{PScondition}
Let $f:D\times \RR\rightarrow \RR$ be a function such that $f(x,\cdot)$ is continuous for every $x\in W$ and $\alpha\in L^1(D)$. Moreover, fixed $\lambda>0,$ let $\mathcal J_\lambda$ be the energy functional
defined in \eqref{Fu2}.
If the sequence $(u_k)_k$ is bounded in $W_0^{1,2}(D)$ and
$$
\|\mathcal{J}_{\lambda}'(u_k)\|_{(W_0^{1,2}(D))^*}\to0\quad \mbox{as}\,\,k\rightarrow +\infty,
$$
then $(u_k)_k$ has a Cauchy subsequence in $W_0^{1,2}(D)$ and so $(u_k)_k$ has a convergent subsequence.
\end{proposition}
\begin{proof}
Let $(u_k)_k$ be a bounded sequence in $W_0^{1,2}(D)$ and let $\lambda>0$ be fixed. By Proposition \ref{theorem 7} there exist a subsequence, which we still denote by $(u_k)_k$ and a function $u_\infty\in L^{\infty}(D)$ such that $u_k\rightarrow u_\infty$ in $L^{\infty}(D)$ as $k\rightarrow +\infty$. By \eqref{eq21}, since $\displaystyle|u(x)|\leqslant \frac{1}{\mu_0\sqrt{\lambda_1}} \|u\|$ for every $x\in D,$ it follows that
\begin{equation*}
\begin{split}
\|u_j-u_l\| & = \sup_{\|v\|\leq 1} |\langle u_j-u_l,v \rangle|\\
& = \sup_{\|v\|\leq 1} \Bigg| \lambda\mathcal{J}_{\lambda}'(u_j)(v)-\lambda\mathcal{J}_{\lambda}'(u_l)(v) + \\
&\qquad\qquad\int_D \alpha(x)(u_j(x)-u_l(x))v(x) d\mu - \lambda\int_D (f(x,u_j(x))-f(x,u_l(x)))v(x) d\mu \Bigg| \\
 &\leq \lambda\|\mathcal{J}_{\lambda}'(u_j)\|+\lambda\|\mathcal{J}_{\lambda}'(u_l)\|\\
& \hspace{5mm}+ \frac{\displaystyle\max_{x\in D} |u_j(x)-u_l(x)|}{\mu_0\sqrt{\lambda_1}}   \int_D |\alpha(x)| d\mu \\
& \hspace{5mm}+ \frac{\lambda}{\mu_0\sqrt{\lambda_1}}  \int_D |f(x,u_j(x))-f(x,u_l(x))| d\mu\rightarrow 0 \quad \mbox{as}\,\,j,l\rightarrow +\infty,
\end{split}
\end{equation*}
since $\mathcal{J}_{\lambda}'(u_k)\rightarrow 0$ as $k \rightarrow +\infty$ and by using the classical Lebesgue Dominated Convergence Theorem.
\end{proof}

Before proving Theorem~\ref{MolicaBisciPrincipal} it will be useful to define another norm on $W_0^{1,2}(D)$
as follows:
\begin{equation}\label{eqnorm}
\|u\|_{\alpha}:=\sqrt{\displaystyle \langle u,u \rangle -\int_D{\alpha(x)}|u(x)|^{2}d\mu},
\end{equation}
where $\alpha$ is the function satisfying the assumptions stated in Theorem~\ref{MolicaBisciPrincipal} and $\langle \cdot,\cdot \rangle$ is defined in \eqref{psc}.
It is easy to see that $\|\cdot\|_\alpha$ is a norm on $W_0^{1,2}(D)$ equivalent to the usual one given in \eqref{eq31}.

More precisely, if $\alpha$ satisfies condition~\eqref{alfa1} we have that
\begin{equation}\label{normaalfa1}
\begin{aligned}
\|u\|_{\alpha}^2=\displaystyle \langle u,u \rangle-\int_D{\alpha(x)}|u(x)|^{2}d\mu\geq \langle u,u \rangle=\|u\|^2,
\end{aligned}
\end{equation}
and, by \eqref{eq21}, we get
$$\begin{aligned}
\|u\|_{\alpha}^2 & =\displaystyle \langle u,u \rangle-\int_D{\alpha(x)}|u(x)|^{2}d\mu\\
& \leq \displaystyle \langle u,u \rangle-\frac{1}{\mu_0^2\lambda_1}\|u\|^2\int_D{\alpha(x)}d\mu\\
& = \Big(1+\frac{1}{\mu_0^2\lambda_1}\int_D{|\alpha(x)|}d\mu\Big)\|u\|^2\,.
\end{aligned}$$

On the other hand, if $\alpha$ 
satisfies condition~\eqref{alfa2} we have that
\begin{equation}\label{normaalfa2}
\begin{aligned}
\|u\|_{\alpha}^2 & =\displaystyle \langle u,u \rangle-\int_D{\alpha(x)}|u(x)|^{2}d\mu\\
& \geq \displaystyle \langle u,u \rangle-\int_D{|\alpha(x)|}|u(x)|^{2}d\mu\\
& \geq \displaystyle \langle u,u \rangle-\frac{1}{\mu_0^2\lambda_1}\|u\|^2\int_D|\alpha(x)|d\mu\\
& = \Big(1-\frac{1}{\mu_0^2\lambda_1}\int_D{|\alpha(x)|}d\mu\Big)\|u\|^2
\end{aligned}
\end{equation}
and
$$\begin{aligned}
\|u\|_{\alpha}^2 & =\displaystyle \langle u,u \rangle-\int_D{\alpha(x)}|u(x)|^{2}d\mu\\
& \leq \displaystyle \langle u,u \rangle-\int_D{|\alpha(x)|}|u(x)|^{2}d\mu\\
& \leq \displaystyle \langle u,u \rangle+\frac{1}{\mu_0^2\lambda_1}\|u\|^2\int_D{|\alpha(x)|}d\mu\\
& \leq 2\|u\|^2\,,
\end{aligned}$$
thanks to \eqref{alfa2}.

\section{Existence and multiplicity results}\label{EMresult}

\noindent This section is devoted to proving our main results.\par
\smallskip
{\bf Proof of Theorem~\ref{MolicaBisciPrincipal}}  The idea of the proof consists of applying Theorem~\ref{Pucci-Serrin+Ricceri} to the functional
$$\mathcal J_\lambda(u)=\frac{1}{2\lambda}\Phi(u)-\Psi(u)\,,$$
where
$$
\Phi(u):=\|u\|_{\alpha}^2,
$$
as well as
$$
\Psi(u):=\displaystyle\int_D F(x,u(x))d\mu,
$$
for any $u\in W_0^{1,2}(D)$. Note that here we perform Theorem~\ref{Pucci-Serrin+Ricceri} by taking the parameter ${\displaystyle\mu=\frac{1}{2\lambda}}$.

First, let us consider the regularity assumptions required on $\Phi$ and $\Psi$. It is easily shown that the functional $\Phi$ is sequentially weakly lower semicontinuous and coercive in $W_0^{1,2}(D)$.

Now, let us prove that the functional $\Psi$ is sequentially weakly continuous in $W_0^{1,2}(D)$. For this purpose, let $(u_k)_k$ be a sequence in the Sobolev space $W_0^{1,2}(D)$ such that
$$u_k \to u_\infty\quad \mbox{weakly in}\,\, W_0^{1,2}(D)$$
as $k\to +\infty$, for some $u_\infty\in W_0^{1,2}(D)$. Consequently, Proposition \ref{theorem 7} yields
$$u_k \to u_\infty\quad \mbox{in}\,\, L^{\infty}(D)\,,$$
that is
\begin{equation}\label{normalinfty}
\|u_k - u\|_\infty\to 0
\end{equation}
as $k\to +\infty$.
Now, by \eqref{normalinfty} there exists a real constant $c>0$ such that
\begin{equation}\label{norma<K}
\|u_k\|_\infty\leq c\quad \mbox{and} \quad \|u_\infty\|_\infty\leq c\,\,\, \mbox{for any}\,\, k\in \NN.
\end{equation}

Then  by \eqref{norma<K} and keeping in mind that $f(x,\cdot)$ is continuous in $\RR$ for every $x\in W$, one has
\begin{equation}\label{psiweakly}
\begin{aligned}
\Big|\Psi(u_k)-\Psi(u_\infty)\Big| & =\Big|\int_D F(x, u_k(x))d\mu- \int_D F(x, u_\infty(x))d\mu\Big|\\
& \leq  \int_D \Big|F(x, u_k(x))- F(x, u_\infty(x))\Big|d\mu\\
& = \int_D \Big|\int_{u_k(x)}^{u_\infty(x)}f(x,t)dt\Big|d\mu\\
& \leq \int_D \Big|\int_{u_k(x)}^{u_\infty(x)}|f(x,t)|dt\Big|d\mu\\
& \leq \int_D \Big|\int_{u_k(x)}^{u_\infty(x)}{\displaystyle \max_{|t|\leq c}|f(x,t)|}dt\Big|d\mu\\
& = \int_D {\displaystyle \max_{|t|\leq c}
|f(x,t)|}\,\,|u_k(x) - u_\infty(x)| d\mu\\
& \leq {\displaystyle\max_{{\footnotesize\begin{array}{c}
x\in D\\
|t|\leq c
\end{array}}}|f(x,t)|}\,\,\|u_k - u_\infty\|_\infty\mu(D).
\end{aligned}
\end{equation}
Hence, by \eqref{normalinfty} and \eqref{psiweakly} we obtain that
$$\Big|\Psi(u_k)-\Psi(u_\infty)\Big|\to 0$$
as $k\to +\infty$, so that the functional $\Psi$ is sequentially weakly continuous in $W_0^{1,2}(D)$ as claimed. Now, we observe that
\begin{equation}\label{junboundedX}\mbox{the functional $\mathcal J_\lambda$ is unbounded from below in $W_0^{1,2}(D)$.}
\end{equation}
Indeed, assumption~\eqref{f2} implies that there exist two positive constants $b_1$ and $b_2$ such that
\begin{equation}\label{inequal}
F(x,t)\geq b_1|t|^{\beta}-b_2\quad \mbox{for any}\,\, x\in D\, \mbox{and}\,\, t\in \RR.
\end{equation}
To prove \eqref{inequal} let $r_0>0$ be as in \eqref{mu0}: then, for every $x\in D$ and for
any $t\in \RR$ with $|t|\geq r_0>0$
$$\frac{t\,f(x,t)}{F(x,t)}\geq \beta.$$
Suppose $t>r_0$. Dividing by $t$ and integrating both terms in $[r_0,t]$
we obtain
$$
F(x,t)\geq \frac{F(x,r_0)}{r^\beta_0}\,|t|^\beta.
$$
Using the same arguments it is easy to prove that if $t<-r_0$ then it results
$$
F(x,t)\geq \frac{F(x,-r_0)}{r^\beta_0}\,|t|^\beta,
$$
so that for any $t\in \RR$ with $|t|\geq r_0$ we get
\begin{equation}\label{Flog}
F(x,t)\geq m(x)\,|t|^\beta,
\end{equation}
where $$m(x):=\frac{\min\{F(x,r_0),\, F(x, -r_0)\}}{r_0^{\beta}}.$$ Since the function $t\mapsto F(\cdot, t)$ is continuous in $\RR$, by
the Weierstrass Theorem, it is bounded for any $t\in \RR$ such that
$|t|\leq r_0$, say
\begin{equation}\label{maxM}
|F(x,t)|\leq \widetilde M(x) \quad \mbox{in}\,\,\, \{|t|\leq r_0\},
\end{equation}
where $$\widetilde M(x):=\max\{|F(x,t)| : |t|\leq r_0\}.$$ Formula~\eqref{inequal}
follows from \eqref{Flog} and \eqref{maxM} by taking

$$b_1:=\min_{x\in D}m(x) \quad \mbox{ and } \quad b_2:=\max_{x\in D}(\widetilde M(x)+m(x)r^\beta_0).$$

\par
Thus, by \eqref{inequal} for any $u\in W_0^{1,2}(D)$ one has
\begin{equation}\label{inequal2}
\int_DF(x,u(x))d\mu\geq b_1\int_D|u(x)|^{\beta}d\mu-b_2\mu(D)\,.
\end{equation}
Let $u_0\in W_0^{1,2}(D)$ with ${\displaystyle \int_D|u_0(x)|^{\beta}d\mu>0}$. Then  by \eqref{inequal2} we have that
$$\begin{aligned}
\mathcal J_\lambda(t u_0) &= \frac{t^2}{2\lambda}\|u_0\|_\alpha^2-\int_D F(x,tu_0(x))d\mu\\
&\leq \frac{t^2}{2\lambda}\|u_0\|_\alpha^2-b_1|t|^{\beta}\int_D|u_0(x)|^{\beta}d\mu+b_2\mu(D)\\
&\rightarrow -\infty,
\end{aligned}$$
as $t\rightarrow +\infty$, since $\beta>2$ by assumption~\eqref{f2} and ${\displaystyle \int_D|u_0(x)|^{\beta}d\mu>0}$\,. The proof of \eqref{junboundedX} is completed.

Now, it remains to see that the functional $\mathcal J_\lambda$ satisfies $(\rm PS)$-condition. To this goal, it is sufficient to employ Proposition~\ref{PScondition}.

More precisely, suppose that there exists a real constant $c>0$ such that $|\mathcal J_\lambda(u_k)|\leq c$ for every $k\in \NN$ and $\mathcal J_\lambda'(u_k)\rightarrow 0$ as $k\rightarrow +\infty$. Now, by \eqref{Fu2} it follows that
\begin{equation}\label{eq3.11}
\begin{split}
c+\beta ^{-1} \|u_k\| & \geq \lambda\mathcal J_\lambda(u_k)-\lambda\beta ^{-1}\mathcal J_\lambda'(u_k) (u_k)\\
& = \left( \frac{1}{2}-\frac{1}{\beta}\right) \|u_k\|^2 + \lambda\beta ^{-1} \int_D (f(x,u_k(x))u_k(x)-\beta F(x,u_k(x))) d\mu \\
& = \left( \frac{1}{2}-\frac{1}{\beta}\right) \|u_k\|^2\\
& \hspace{5mm}+ \lambda\beta ^{-1} \int_{\{x\in D :|u_k(x)|<r_0\}} (f(x,u_k(x))u_k(x)-\beta F(x,u_k(x))) d\mu \\
& \hspace{5mm}+ \lambda\beta ^{-1}\int_{\{x\in D : |u_k(x)|\geq r_0\}} (f(x,u_k(x))u_k(x)-\beta F(x,u_k(x))) d\mu,
\end{split}
\end{equation}
 for every $k\in \NN$ sufficiently large.\par
By \eqref{f2}, the third term in (\ref{eq3.11}) is non-negative while the second term is bounded by a positive constant independent of $k$. Since $\beta> 2$, (\ref{eq3.11}) implies that $(u_k)_k$ is bounded in $W^{1,2}_0(D)$. Therefore, by Proposition \ref{PScondition}, the energy functional $\mathcal J_\lambda$ satisfies the (PS) compactness condition.\par
Moreover, let $\varrho>0$ and
$$\chi(\varrho):=\displaystyle\inf_{u\in \mathbb{B}_{\varrho}}\frac{\displaystyle\sup_{v\in \mathbb{B}_{\varrho}}\Psi(v) -\Psi(u)}{\varrho-\|u\|_{\alpha}^2}\,,$$
where
$$\mathbb{B}_{\varrho}=\Big\{v\in W_0^{1,2}(D) : \|v\|_\alpha<\sqrt \varrho\Big\}\,.$$

The definition of $\chi$ yields that for every $u\in \mathbb{B}_{\varrho}$
$$\chi(\varrho) \leq \frac{\displaystyle\sup_{v\in \mathbb{B}_{\varrho}}\Psi(v) -\Psi(u)}{\varrho-\|u\|_{\alpha}^2}$$
thus, being $0\in \mathbb{B}_\varrho$, we obtain that
\begin{equation}\label{Funzionale1}
\begin{aligned}
\chi(\varrho)
& \leq \frac 1 \varrho\displaystyle\sup_{v\in \mathbb{B}_{\varrho}}\Psi(v)\\
& \leq \frac 1 \varrho \displaystyle\sup_{v\in \overline{\mathbb{B}}_\varrho}\left|\int_D F(x,v(x))d\mu\right|\\
& \leq \frac 1 \varrho \displaystyle\sup_{v\in \overline{\mathbb{B}}_\varrho}\int_D \left|F(x,v(x))\right| d\mu.
\end{aligned}
\end{equation}

Now, assume that the function $\alpha$ satisfies assumption~\eqref{alfa1}. Then  if $v\in \overline{\mathbb{B}}_\varrho$, by \eqref{eq21} and \eqref{normaalfa1} we get that
\begin{equation}\label{valfa1}
|v(x)|\leq \frac{1}{\mu_0\sqrt{\lambda_1}}\|v\|\leq \frac{1}{\mu_0\sqrt{\lambda_1}}\|v\|_\alpha\leq \frac{\sqrt{\varrho} }{\mu_0\sqrt{\lambda_1}}\quad \mbox{for any}\,\, x\in D,
\end{equation}
which, together with the continuity of $F$ and the finiteness of $D$, gives for any $x\in D$
\begin{equation}\label{Funzionale2F}
\left|F(x,v(x))\right|\leq \displaystyle\max_{{\footnotesize\begin{array}{c}
x\in D\\
|s|\leq \displaystyle\frac{\sqrt{\varrho} }{\mu_0\sqrt{\lambda_1}}
\end{array}}}\left|\int_0^{s}f(x,t)dt\right|.
\end{equation}
Therefore, inequality~\eqref{Funzionale2F} yields
\begin{equation}\label{Funzionale2F2}
\int_D \left|F(x,v(x))\right|d\mu\leq \mu(D)\displaystyle\max_{{\footnotesize\begin{array}{c}
x\in D\\
|s|\leq \displaystyle\frac{\sqrt{\varrho} }{\mu_0\sqrt{\lambda_1}}
\end{array}}}\left|\int_0^{s}f(x,t)dt\right|
\end{equation}
for any $v\in \overline{\mathbb{B}}_{\varrho}$.

By \eqref{Funzionale1} and \eqref{Funzionale2F2} we have that
$$\chi(\varrho)\leq \frac 1 \varrho \displaystyle\max_{{\footnotesize\begin{array}{c}
x\in D\\
|s|\leq \displaystyle\frac{\sqrt{\varrho} }{\mu_0\sqrt{\lambda_1}}
\end{array}}}\left|\int_0^{s}f(x,t)dt\right|<\frac{1}{2\lambda}\,,$$
provided $\lambda$ satisfies condition \eqref{lambda}.

On the other hand, if the function $\alpha$ satisfies assumption~\eqref{alfa2}, we can argue in the same way, just replacing \eqref{valfa1} with the following inequality
\begin{equation}\label{valfa2}
\begin{aligned}
|v(x)| & \leq \frac{1}{\mu_0\sqrt{\lambda_1}}\|v\|\\
& \leq {\displaystyle \frac{1}{\mu_0\sqrt{\lambda_1}\sqrt{1-\displaystyle\frac{1}{\mu_0^2\lambda_1}{\displaystyle\int_D{|\alpha(x)|}d\mu}}}}\|v\|_\alpha\\
& \leq {\displaystyle \frac{1}{\mu_0\sqrt{\lambda_1}\sqrt{1-\displaystyle\frac{1}{\mu_0^2\lambda_1}{\displaystyle\int_D{|\alpha(x)|}d\mu}}}}\sqrt{\varrho}
\end{aligned}
\end{equation}
for any $x\in D$, and thanks to \eqref{normaalfa2}.

In both cases, owing to Theorem~\ref{Pucci-Serrin+Ricceri} and considering \eqref{f0} and \eqref{junboundedX}, we conclude that problem~\eqref{Np0}
admits at least two non-trivial weak solutions one of which lies in $\mathbb{B}_\varrho$.
The proof of Theorem~\ref{MolicaBisciPrincipal} is finally complete on account of Proposition~\ref{strong}.\hfill $\Box$

\medskip

\begin{remark}\label{notaMorrey}
{\rm First, we notice that condition~\eqref{eq21} plays a crucial role in the proof of Theorem~\ref{MolicaBisciPrincipal}, whereas the Sobolev embedding theorems are employed in the classical case of bounded domains; see, among others, the papers \cite{ar, rabinowitz, struwe}. The proof of Theorem~\ref{MolicaBisciPrincipal} relies on the same arguments as used in \cite[Theorem 1]{MRS}.
}
\end{remark}

\begin{remark}\label{corollarioprova2}
\rm{Note that the trivial function
is a weak solution of problem~\eqref{Np0} if and only if $f(\cdot, 0)= 0$. Hence, condition~\eqref{f0} assures that all the solutions of problem~\eqref{Np0}, if any, are non-trivial. In the case when $f(\cdot,0) = 0$, in order to get the existence of multiple solutions for \eqref{Np0} some extra assumptions on the nonlinear term $f$ are necessary.}
\end{remark}
{\bf Proof of Theorem~\ref{MolicaBisciSpecial}} The conclusion is achieved by using Theorem~\ref{MolicaBisciPrincipal}. Indeed,
as it is easily seen, condition~\eqref{f1} yields
\begin{equation}\label{Funzionale2F3}
\displaystyle\max_{(x,s)\in D\times[-M_0,M_0]}\left|\int_0^{s}f(x,t)dt\right|\leq \mu_0^2\frac{M_0^2}{(\sigma+1)}\frac{\lambda_1}{2}.
\end{equation}

\noindent Thus, if we set
\begin{equation*}
\theta:=\left\{\begin{array}{ll}
\displaystyle \frac{\mu_0^2M_0^2\lambda_1}{2\displaystyle\max_{(x,s)\in D\times[-M_0,M_0]}\left|\int_0^{s}f(x,t)dt\right|} \ & \ \mbox{if} \displaystyle\max_{(x,s)\in D\times[-M_0,M_0]}\left|\int_0^{s}f(x,t)dt\right|>0\\
\\
+\infty \ & \ \mbox{otherwise},
\end{array}\right.
\end{equation*}
by the fact that $\sigma>0$, one has
\begin{equation}\label{FUno}
1<\sigma+1\leq \theta\leq \lambda^*,
\end{equation}
since \eqref{Funzionale2F3} holds and
$$
\lambda^*:=\sup_{\varrho>0}\frac{\varrho}{\displaystyle\max_{{\footnotesize\begin{array}{c}
x\in D\\
|s|\leq \kappa\sqrt{\varrho}
\end{array}}}\left|\int_0^{s}f(x,t)dt\right|}= \frac{1}{\kappa^2}\sup_{z>0}\frac{z^2}{\displaystyle\max_{{\footnotesize\begin{array}{c}
x\in D\\
|s|\leq z
\end{array}}}\left|\int_0^{s}f(x,t)dt\right|},
$$
on account of the definition given in \eqref{struttela}.\par
\noindent Then  thanks to \eqref{FUno}, Theorem~\ref{MolicaBisciPrincipal} ensures that for $\lambda=1$ problem~\eqref{Np}
admits at least two non-trivial weak solutions one of which lies in the open ball $\mathbb{B}_{\mu_0^2M_0^2\lambda_1}$. Finally, Proposition~\ref{strong} ensures that every weak solution of problem~\eqref{Np} is also classical and this concludes the proof of Theorem~\ref{MolicaBisciSpecial}.\hfill $\Box$\par

\begin{remark}
\rm{We notice that a condition similar to \eqref{f1} has been used previously in the literature by K.J. Falconer and J. Hu by studying elliptic problems on the gasket; see \cite[Theorem 3.5]{Falconer}}.
\end{remark}

\medskip
{\bf Proof of Theorem~\ref{MolicaBisciSpecial2}} \rm{In order to prove the existence of a non-negative
solution of problem~\eqref{Np} it is enough to introduce the
functions
$${\displaystyle F_+(x,t):=\int_0^t  f_+(x,\tau)d\tau},$$
with
$$\begin{array}{ccc}
f_+(x,t)=\begin{cases}
f(x,t) & \mbox{if} \,\,\, t\geq 0\\
0 & \mbox{if} \,\,\, t<0
\end{cases},
\end{array}$$
for every $x\in D$.\par
\noindent Note that $f_+$ satisfies condition~\eqref{f1}, while assumption~\eqref{f2} is satisfied  by $f_+$ and
$F_+$ for every $x\in D$ and for any $t>r_0$.

\noindent Thus, let $\mathcal J_+: W_0^{1,2}(D) \to \RR$ be the functional defined as
follows
\begin{equation}
\begin{aligned}
\mathcal J_+(u)&:=\frac{1}{2}\|u\|^2-\frac{1}{2}\int_D {\alpha(x)}|u(x)|^{2}\,d\mu
-\int_D F_+(x, u(x))d\mu\\
&=\frac{1}{2}\|u\|^2-\frac{1}{2}\int_D {\alpha(x)}|u(x)|^{2}\,d\mu
-\int_D F(x, u^+(x))d\mu,
\end{aligned}
\end{equation}
where $u^+(x):=\max\{0,u(x)\}$ for every $x\in D$.\par
\noindent It is easy to see that the energy functional $\mathcal J_+$ is well
defined and Fr\'echet differentiable in $u\in W_0^{1,2}(D)$. More precisely
\begin{equation}\label{test+}
\begin{aligned}
\langle \mathcal{J}'_{+}(u), v\rangle = {\langle u,v \rangle}-\int_D\alpha(x)u(x)v(x)d\mu
-\int_D f_+(x,u(x))v(x)d\mu,
\end{aligned}
\end{equation}
for any
$v\in W_0^{1,2}(D)$.\par
 \noindent Now, by \eqref{f1l} it follows that
\begin{equation}\label{f1N}
\begin{aligned}
& \mbox{there are positive constants}\,\,\, M_0\,\, \mbox{and}\,\, \sigma\,\, \mbox{such that}\\
& \,\,\,\,\,\,\,\displaystyle\max_{(x,s)\in \mathop D\limits^ \circ\times[-M_0,M_0]}f_+(x,s)\leq \mu_0^2\frac{M_0}{(\sigma+1)}\frac{\lambda_1}{2}.
\end{aligned}
\end{equation}
Since \eqref{f2new} holds, as in the proof of Theorem~\ref{MolicaBisciSpecial}, there exists $u_\infty\in W^{1,2}_0(D)\setminus\{0\}$ that classically solves the following Dirichlet problem
 \begin{equation}\label{Npps}
\left\{
\begin{array}{l}
-\Delta_{\mu} u(x)=\alpha(x)u(x)+ f_+(x,u(x))\quad x \in \mathop D\limits^ \circ\\
\medskip
\,\, u|_{\partial D}=0.\\
\end{array}
\right.
\end{equation}

\noindent We claim that $u_\infty$ is non-negative in $D$. Indeed, as usual, let $u^-_\infty:=\max\{-u_{\infty}, 0\}$ be the negative part
of $u_\infty$.
Since $$\Gamma(u^-_\infty,u_\infty)=\Gamma(u^-_\infty)+\Gamma(u^-_\infty,u^+_\infty)\geq |\nabla u^-_\infty|^2,$$ and $\alpha\in L^1(D)$ satisfy \eqref{alfa1}, on account of \eqref{parti} it follows that

\begin{equation*}
\begin{aligned}
\int_{D} |\nabla u^-_\infty|^2(x) d\mu &\leq \int_D \Gamma(u^-_\infty)(x)d\mu+\int_D\Gamma(u^-_\infty,u^+_\infty)(x)d\mu\\
&=\int_{D} \Gamma(u^-_\infty,u_\infty)(x) d\mu = - \int_{D} u^-_\infty(x) \Delta_\mu u_\infty(x) d\mu\\
& = \int_{\Omega} u^-_\infty(x) f_+(x,u^+_\infty(x)) dx=0.
\end{aligned}
\end{equation*}
This implies that $u^-_\infty \equiv 0$ and thus $u_\infty\geq 0$ in $D$. Since $f^+=f$ in $\RR^+_0:=[0,+\infty[$ the function $u_\infty\in W^{1,2}_0(D)\setminus\{0\}$ classically solves \eqref{Np} as claimed. \hfill$\Box$
}
\smallskip

\begin{remark}
We notice that the functional $\mathcal J_+$ defined along the proof of Theorem~\ref{MolicaBisciSpecial2} satisfies the compactness (PS) condition, so that classical variational arguments can be applied by studying problem \eqref{Npps}. To this goal let us observe that, arguing as in the proof of Theorem~\ref{MolicaBisciPrincipal}, assumption~\eqref{f2new} implies that there exist two positive constants $b_1$ and $b_2$ such that
\begin{equation}\label{inequalNew}
F(x,t)\geq b_1t^{\beta}-b_2\quad \mbox{for any}\,\, t\geq r_0\, \mbox{and every}\,\, x\in D.
\end{equation}
Let $c\in \RR$ and let $(u_k)_k$ be a
sequence in $W^{1,2}_0(D)$ such that
\begin{equation}\label{Jc0}
\mathcal J_+(u_k)\to c,
\end{equation}
and
\begin{equation}\label{J'00}
\sup\Big\{ \big|\langle\,\mathcal J'_+(u_k),v\,\rangle \big|\,: \;
v\in
W^{1,2}_0(D)\,, \|v\|=1\Big\}\to 0,
\end{equation}
as $k\to +\infty$. Without loss of generality, we only consider the case of a (definitively) non-trivial sequence.
For any $k\in \NN$ by \eqref{J'00} and \eqref{Jc0} it easily follows
that there exists $\kappa>0$ such that
\begin{equation}\label{j'limitato0}
\Big|\langle \mathcal J'_+(u_k), \frac{u_k}{\|u_k\|}\rangle\Big| \leq \kappa\,,
\end{equation}
and
\begin{equation}\label{jlimitato0}
|\mathcal J_+(u_k)|\leq \kappa,
\end{equation}
for every $k\in \NN$.\par
\noindent Moreover, by \eqref{f02} it follows
that
\begin{equation}\label{uj<r}
\begin{aligned} & \Big|\int_{\{x\in D:\,0\leq u_k^+(x)\leq
\,r_0\}}\Big(F(x, u_k^+(x))-\frac 1 \beta\, f(x,
u_k^+(x))\, u_k^+(x)\,\Big)\,dx\Big|\\
& \qquad \qquad \qquad \leq \sum_{x\in \{x\in D:\,0\leq u_k^+(x)\leq
\,r_0\}} \mu (x)\Big(F(x, u_k^+(x))-\frac 1 \beta\, f(x,
u_k^+(x))\, u_k^+(x)\,\Big)\\
&\qquad \qquad \qquad\leq
\sum_{x\in \{x\in D:\,0\leq u_k^+(x)\leq
\,r_0\}} \mu (x)\max_{(x,t)\in D\times [0,r_0]}\Big(F(x, t)-\frac 1 \beta\, f(x,t)\, t\,\Big)\\
&\qquad \qquad \qquad\leq \mu(D)\max_{(x,t)\in D\times [0,r_0]}\Big(F(x, t)-\frac 1 \beta\, f(x,t)\, t\,\Big)
<+\infty.
\end{aligned}
\end{equation}

\noindent Also, thanks to \eqref{mu0} and \eqref{uj<r} we get
\begin{equation}\label{jj'0}
\begin{aligned}
& \mathcal J_+(u_k)-\frac 1 \beta \langle \mathcal J'_+(u_k), u_k\rangle\\
& \qquad \qquad = \left(\frac 1 2 -\frac 1 \beta\right)\|u_k\|^2-\frac 1 \beta \int_D \Big(\beta F_+(x, u_k(x))-f_+(x, u_k(x)) \,u_k(x)\Big)\,dx\\
& \qquad \qquad \geq \left(\frac 1 2 -\frac 1 \beta\right)\|u_k\|^2\\ & \qquad \qquad \qquad \qquad -\int_{\{x\in D:\,0\leq u_k(x)\leq r_0\}}\Big(F(x, u_k^+(x))-\frac 1 \beta\, f(x, u_k^+(x))\,u_k^+(x)\Big)\,dx\\
& \qquad \qquad \geq \left(\frac 1 2 -\frac 1 \beta\right)\|u_k\|^2-\mu(D)\max_{(x,t)\in D\times [0,r_0]}\Big(F(x, t)-\frac 1 \beta\, f(x,t)\, t\,\Big).
\end{aligned}
\end{equation}
As a consequence of \eqref{j'limitato0} and  \eqref{jlimitato0} we also have
$$\mathcal J_+(u_k)-\frac 1 \beta \langle \mathcal J'_+(u_k), u_k\rangle\leq \kappa \left(1+ \|u_k\|\right)$$
so that, by \eqref{jj'0} for any $k\in \NN$
$$\|u_k\|^2 \leq \kappa_*\left(1+\|u_k\|\right)$$
for a suitable constant $\kappa_*>0$.
The above inequality immediately yields that the sequence $(u_k)_k$ is bounded in $W^{1,2}_0(D)$. Therefore, by Proposition \ref{PScondition}, the energy functional $\mathcal J_+$ satisfies the (PS) condition as claimed.
\end{remark}

{\bf Proof of Theorem \ref{MolicaBisciSpecial2C}} Since $p>2$, one clearly has
$$
\limsup_{t\rightarrow 0^+}\frac{|t|^{p-2}t}{t}=0.
$$
Moreover, it easily seen that condition \eqref{f2new} is also satisfied.
Consequently, arguing as in the proof of Theorem \ref{MolicaBisciSpecial2}, there exists a function $u_\infty\in W_0^{1,2}(D)\setminus\{0\}$ such that
\begin{equation}\label{1.1}
-\Delta_\mu u_\infty(x) - \gamma u_\infty(x)=(u^+_\infty(x))^{p-1} \qquad {\rm in} \quad \mathop D\limits^ \circ.
\end{equation}

\noindent Let us prove that $u_\infty\geq 0$ in $D$. To this aim, let $u^-_\infty(x) := \min \{u_\infty(x),0\}$ for every $x\in D$. Since $u^-_\infty(x) u^+_\infty(x) = u^-_\infty(x)(u^+_\infty(x))^{p-1} = 0$ for every $x\in D$, we have
\[
- \int_{D} u^-_\infty(x) \Delta_\mu u_\infty(x) d\mu - \gamma \int_{D} (u^-_\infty(x))^2 d\mu = 0 \,.
\]
Since $u_\infty=u^+_\infty + u^-_\infty$, the above equation leads to
\begin{equation}\label{1.2}
\begin{split}
\frac{\gamma}{\lambda_1 } \int_{D} |\nabla u^-_\infty|^2(x) d\mu &\geq \gamma \int_{D} (u^-_\infty(x))^2 d\mu\\
& = \int_{D} |\nabla u^-_\infty|^2(x) d\mu - \int_{D} u^-_\infty(x) \Delta_\mu u^+_\infty(x) d\mu.
\end{split}
\end{equation}
Note that
\begin{equation}\label{1.3}
\begin{split}
-\int_{D} u^-_\infty(x) \Delta_\mu u^+_\infty(x) d\mu & = -\sum_{x\in \mathop D\limits^ \circ} u^-_\infty(x) \sum_{y\sim x} w(x,y) (u^+_\infty(y)-u^+_\infty(x))\\
& = -\sum_{x\in \mathop D\limits^ \circ} \sum_{y\sim x} w(x,y) u^-_\infty(x) u^+_\infty(y) \geq 0 \,.
\end{split}
\end{equation}
Inserting (\ref{1.3}) into (\ref{1.2}) and recalling that $\gamma < \lambda_1$, we obtain $\displaystyle\int_{D} |\nabla u^-_\infty|^2(x) d\mu = 0$, which implies that $u^-_\infty \equiv 0$ in $D$. Whence $u_\infty\geq 0$ in $D$ and (\ref{1.1}) becomes
\begin{equation}\label{1.4}
\left\{
\begin{array}{l}
-\Delta_{\mu} u_\infty(x)-\gamma u_\infty(x)=u_\infty(x)^{p-1}\quad x \in \mathop D\limits^ \circ\\
\medskip
\,\,u_\infty(x)\geq 0 \quad x\in \mathop D\limits^ \circ\\
\medskip
\,\, u_\infty|_{\partial D}=0.
\end{array}
\right.
\end{equation}
\noindent Finally, suppose that
$$u_\infty(x_0)=0=\displaystyle\min_{x\in D} u_\infty(x)$$ for some $x_0\in \mathop D\limits^ \circ$. If $y\in D$ is adjacent to $x_0$ then $\Delta_\mu u_\infty(x_0)=0$ by \eqref{1.4}. The definition of the $\mu$-Laplacian immediately yields $u_\infty(y)=0$. Therefore we conclude
that $u_\infty\equiv 0$ in $D$, which is absurd. Thus $u_\infty>0$ in $\mathop D\limits^ \circ$ as claimed.
The proof is now complete.\hfill $\Box$

\medskip
\begin{remark}\label{RemFinale}
\rm{We notice that condition \eqref{NpXXss} is a special case of \eqref{f1}.
Consequently, it is easily seen that Theorems \ref{MolicaBisciSpecial} and \ref{MolicaBisciSpecial2} immediately yield the conclusions of Corollary \ref{MolicaBisciSpecialSpecial} given in Introduction.}
\end{remark}

\section{Yamabe-type problems on graphs}\label{YG}

Let $m\in \NN$ and $p\in \RR$ with $p>1.$ Denote by $W^{m,p}_0(D)$ be the completion of the space
$$
C_0^m(D):=\{u:D\rightarrow\RR: |\nabla^{j} u|=0\mbox{ on } \partial D, \mbox{ for every}\, j=0,...,m-1\}
$$
with respect to the Sobolev norm
\begin{equation}\label{nu0G}
 \|u\| := \sum_{k=0}^{m}\||\nabla^{k} u|\|_{L^p(D)},
 \end{equation}
where, for every $x\in D$
\begin{equation*}\label{problem1G}
| \nabla^k u|(x) :=\begin{cases}
\displaystyle|\nabla \Delta^{(m-1)/2}u|(x) &\mbox{if } k\in 2\NN+1\\\\
\,\displaystyle |\Delta^{m/2}u|(x) &\mbox{if } k\in 2\NN,
  \end{cases}
\end{equation*}
with $k=0,...,m,$
and
\begin{equation*}\label{problem1G2}
\Delta_\mu^{\ell} u(x) := \begin{cases}
\displaystyle \,\,u(x) &\mbox{if } \ell=0\\\\
\,\displaystyle \frac{1}{\mu (x)} \sum_{y\sim x} w(x,y) (\Delta_\mu^{\ell-1} u(y)-\Delta_\mu^{\ell-1}u(x)) &\mbox{if } \ell>0,
  \end{cases}
\end{equation*}
for every $x\in D$. By \cite[Theorem 27]{GLY2}
the space $(W^{m,p}_0(D),\|u\|_{W^{m,p}_0(D)})$ is a finite dimensional Banach space,
where
$$
\|u\|_{W^{m,p}_0(D)}:=\||\nabla^mu|\|_{L^p(D)}
$$
is a norm on $W^{m,p}_0(D)$ equivalent to $\|\cdot\|$ given in \eqref{nu0G}.\par
With the above notations, a general version of Theorem~\ref{MolicaBisciPrincipal} proved in Section \ref{SezA} can be achieved for the following Yamabe-type problem
\begin{equation}\label{NpXXgeneral}
\left\{
\begin{array}{l}
\mathcal{L}_{m,p} u(x)=\lambda f(x,u(x))\quad x \in \mathop D\limits^ \circ\\
\smallskip
\smallskip
\,\, |\nabla^{j} u|=0\,\,\ \mbox{ on } \partial D, \mbox{ with }\, j=0,...,m-1,\\
\end{array}
\right.
\end{equation}
\noindent where $\lambda$ is a real positive parameter, $f:D\times\RR\rightarrow \RR$ is a suitable continuous non-linear term and $\mathcal{L}_{m,p}:W^{m,p}_0(D)\rightarrow L^p(D)$ denotes the $(m,p)$-Laplacian operator defined in the distributional sense for every $u\in W^{m,p}_0(D)$ as
\begin{equation}\label{problem1Gigio}
\mathcal{L}_{m,p}  u :=\begin{cases}
\displaystyle\int_{D}|\nabla^mu|^{p-2}(x)\Gamma\left(\Delta^{(m-1)/2}u,\Delta^{(m-1)/2}v \right)(x)d\mu&\mbox{if } m\in 2\NN+1\\\\
\,\displaystyle \int_{D}|\nabla^mu|^{p-2}(x)\Delta^{m/2}u(x)\cdot \Delta^{m/2}v(x)d\mu &\mbox{if } m\in 2\NN,
  \end{cases}
\end{equation}
for any $v\in W^{m,p}_0(D)$.\par
The $(m,p)$-Laplacian $\mathcal{L}_{m,p}$ can be explicitly computed at any point of $x\in D$. In particular,
$\mathcal{L}_{m,2}$ is the poly-Laplacian operator $(-\Delta_\mu)^m$ that, for $p=2$, reduces to the $\mu$-Laplacian defined in \eqref{probl1}. We emphasize that $p$-Laplacian equations on graphs also in connection with geometrical analysis problems have been considered in the literature; see, among others the papers \cite{Ge1,Ge2,GeHJ,GeJ1} as well as \cite{GLY1,ZC,ZL1,ZL2}.
Moreover, existence and convergence of solutions for nonlinear biharmonic equations on graphs
is proved in \cite{HSZ} as well as
the 1-Yamabe equation on the graph-theoretical setting has been investigated in \cite{GeJ3}.
\par
Thanks to \eqref{problem1Gigio} the solutions of \eqref{NpXXgeneral} are the critical points of the energy functional $\mathcal J:W^{m,p}_0(D)\rightarrow\RR$ defined by
\begin{equation}\label{Fu2General}
\begin{aligned}
\mathcal J_\lambda(u)& :=\frac{1}{p}\|u\|_{W^{m,p}_0(D)}^p
-\lambda\int_D F(x, u(x))d\mu,
\end{aligned}
\end{equation}
for every $u\in W^{m,p}_0(D)$.\par
Similarly to \eqref{lamba1}, let us define the real constant
\begin{align}\label{lamba1pp}
\lambda_{m,p}:=\inf_{u\in C_0^m(D)\setminus\{0\}}\frac{\displaystyle \int_{D} |\nabla^m u|^p(x) d\mu}{\displaystyle \int_{D} |u(x)|^p d\mu}.
\end{align}
\indent With the above notations the next result holds.
\begin{theorem}\label{MolicaBisciPrincipalGeneralX}
Let $\mathscr G:=(V,E)$ be a weighted locally finite graph, $D$ be a bounded domain of $V$ such that $\mathop D\limits^ \circ\neq \emptyset$ and $\partial D\neq \emptyset$ and let $\mu:D\rightarrow ]0,+\infty[$ be a measure on $D$.
Let $\mu_0:=\min_{x\in D}\mu(x)>0$ and let
$f:D\times\RR\rightarrow\RR$ be a function such that \eqref{f0} holds
as well as
\begin{equation}\label{f25G}
\begin{aligned}
& \qquad\,\,\qquad \mbox{there are}\,\, \beta>p\,\, \mbox{and}\,\, r_0>0\,\, \mbox{such that}\\
& tf(x,t)\geq \beta F(x,t)>0\,\, \mbox{for any}\,\, |t|\geq r_0\,\,\mbox{and every}\,\, x\in D,
\end{aligned}
\end{equation}
where $F$ is the potential given by \eqref{F}.

Then  for any $\varrho>0$ and any
\begin{equation}\label{lambda}
0<\lambda<\frac{\varrho}{2\displaystyle\max_{{\footnotesize\begin{array}{c}
x\in D\\
|s|\leq \kappa_{m,p}{\varrho^{1/p}}
\end{array}}}\left|\int_0^{s}f(x,t)dt\right|}\,,
\end{equation}
where
\begin{equation}\label{kappaG}
\kappa_{m,p}:=\displaystyle\frac{1}{\mu_0\lambda_{m,p}^{1/p}},
\end{equation}
the problem~\eqref{Np0} admits at least two non-trivial solutions one of which lies in
$$
\mathbb{B}_\varrho^{(m,p)}:=\left\{u\in W^{m,p}_0(D):\displaystyle \|u\|_{W^{m,p}_0(D)} <{\varrho^{1/p}}\right\}.
$$
\end{theorem}

\indent Yamabe equations on infinite graphs have been studied in several papers also in connection with the uniqueness result for a Kazdan-Warner type
problem on bounded domains; see, among others, the papers \cite{GeJ2, GeJ3, GLY2} and very recently in \cite[Corollary 5.3]{piste}.\par
 We emphasize that the existence result given in Theorem \ref{MolicaBisciPrincipalGeneralX} is in the same spirit of \cite[Theorem 1.1]{piste}; see also Theorem \ref{MolicaBisciSpecial2C}.\par
\medskip
\noindent {\bf Acknowledgements.} The manuscript was realized under the auspices of the Italian MIUR project \textit{Variational methods, with applications to problems in mathematical physics and geometry} (2015KB9WPT 009) and the Slovenian Research Agency grants (P1-0292, N1-0114, N1-0083).

\end{document}